\documentclass[12pt]{amsart}
\usepackage{amsfonts,amssymb,latexsym,amsmath, amsxtra}
\usepackage[all]{xy}
\usepackage[dvips]{graphics}
\usepackage{young}
\usepackage[enableskew]{youngtab}

\usepackage[OT2,T1]{fontenc}
\DeclareSymbolFont{cyrletters}{OT2}{wncyr}{m}{n}
\DeclareMathSymbol{\Sha}{\mathalpha}{cyrletters}{"58}

\theoremstyle{plain}
\newtheorem{theorem}{Theorem}[section]

\newtheorem{proposition}[theorem]{Proposition}

\theoremstyle{definition}

\theoremstyle{remark}
\newtheorem*{remark}{Remark}

\numberwithin{equation}{section}

\def\({\left(}
\def\){\right)}
\def\<{\left<}
\def\>{\right>}

\begin{document}

\title[Note on a partition limit]{Note on a partition limit theorem for rank and crank}

\author{Persi Diaconis}
\address{Stanford University, Department of Mathematics and Statistics, Sequoia Hall, 390 Serra Mall, Stanford, CA 94305-4065, USA}
\email{diaconis@math.stanford.edu}

\author{   Svante Janson}
\address{Uppsala University,
   Department of Mathematics,
   Box 480, 751 06,
   Uppsala, Sweden}
 \email{svante.janson@math.uu.se}

\author{Robert C. Rhoades}
\address{Stanford University, Department of Mathematics, Sloan Hall, 450 Serra Mall,, Stanford, CA 94305-2125, USA}
\email{rhoades@math.stanford.edu}


\date{\today}
\thispagestyle{empty} \vspace{.5cm}
\begin{abstract}
If $\lambda$ is a partition of $n$, the rank $rk(\lambda)$ is the size of the largest part minus the number of parts. Under the uniform distribution on partitions, in Bringmann-Mahlburg-Rhoades
\cite{bmr11} it is shown that $rk(\lambda)/\sqrt{6n}$ has a limiting distribution. We identify the limit as the difference between two independent extreme value distributions and as the distribution of $\beta(T)$ where $\beta(t)$ is standard Brownian motion and $T$ is the first time that an independent three-dimensional Brownian motion hits the unit sphere. The same limit holds for the crank.
\end{abstract}

\maketitle


 Let $\lambda=(\lambda_1,\lambda_2,\dots,\lambda_l)$ with $\lambda_1\geq\lambda_2\geq\dots\geq\lambda_l>0$, $\sum_{i=1}^l\lambda_i=n$ be a partition of $n$. The rank $rk(\lambda)=\lambda_1-l$ was introduced by Dyson \cite{dyson} to explain some of Ramanujan's congruences (e.g., $p(5n+4)\equiv0\pmod5$). This story with its many developments to other primes and cranks can be accessed from the Wikipedia entry on Ramanujan's congruences and its references; see, for example, \cite{andrews} for more on the crank. 
 In \cite{bmr11,bmr12}, Bringmann, Mahlburg, and Rhoades studied the moments of $rk(\lambda)$ when $\lambda$ is chosen from the uniform distribution on the partitions of $n$, $\mathcal{P}(n)$. Write $p(n)$ for $|\mathcal{P}(n)|$. Their results imply
\begin{theorem}[\cite{bmr11}]
There is a distribution function $F_r(x)$ on $\mathbb{R}$ such that, for all $x$,
\begin{equation}
\lim_{n\to\infty}\frac1{p(n)}\left|\left\{\lambda\in\mathcal{P}(n):\frac{rk(\lambda)}{\sqrt{6n}}\leq x\right\}\right|\longrightarrow F_r(x).
\label{1}
\end{equation}
\end{theorem}

The main result of this note identifies the limit $F_r(x)$. Recall the extreme value distribution function $F_e(x)=e^{-e^{-x}}$, $-\infty<x<\infty$. This occurs as the limiting distribution of the maximum of a wide variety of sequences of random variables. Indeed, Erd\"os and Lehner
\cite{erdos} proved that it governs the distribution of the largest part of a random partition:
\begin{equation}
\lim_{n\to\infty}\frac1{p(n)}\left|\left\{\lambda\in\mathcal{P}(n):\frac{\lambda_1-\frac{\sqrt{6n}}{\pi}\log\left(\sqrt{6n}/\pi\right)}{\sqrt{6n}/\pi}\leq x\right\}\right|\longrightarrow F_e(x).
\label{2}
\end{equation}
By reflection symmetry, the limiting distribution of $l(\lambda)$ is the same. This makes it plausible that the limiting distribution of $rk(\lambda)$ should be the difference between two independent extreme value distributions. This is our result:
\begin{proposition}\label{prop:result}
The  limit $F_r(x)$ in \eqref{1} is the distribution function of $(W_1-W_2)/\pi$ where $W_i$ are independent with distribution function $F_e(x)$. Thus $F_r(x)=1/(1+e^{-\pi x})$.
\end{proposition}
\begin{remark}
\begin{enumerate}
\item From \eqref{2} both $\lambda_1$ and $l(\lambda)$ are typically of size $\sqrt{n}\log n$. At first it seems surprising that $rk(\lambda)$ is only of size $\sqrt{n}$. On reflection, we see that the additive $\sqrt{n}\log n$ terms cancel in $\lambda_1-l$ and the scaling determines the size. The limit in the proposition is exactly the convolution of the two fluctuation distributions. 
Indeed, Szekeres  \cite{szekeres} has proved a local limit theorem for the joint distribution of $\lambda_1$ and $l(\lambda)$ which shows that they are asymptotically independent.

\item Let $N(n,m)$ be the number of partitions of $n$ with rank equal to $m$. Let $N_k(n)=\sum_{m\in\mathbb{Z}}m^kN(n,m)$ be the $k$th moment. Then $N_0(n)=p(n),\ N_k(n)=0$ for $k$ odd by symmetry ($N(n,m)=N(n,-m)$). Bringmann, Mahlburg, and the third author
\cite{bmr11} prove that, as $n$ tends to infinity, for each $l=1,2,\dots$,
\begin{equation}
\frac{N_{2l}(n)}{(6n)^lp(n)}\sim\left(2^{2l}-2\right)|B_{2l}|
\label{3}
\end{equation}
where $B_{2l}$ is the $2l$th Bernoulli number, $\frac{t}{(e^t-1)}=\sum_{n=0}^\infty B_nt^n/n!$. A routine application of the method of moments yields \eqref{1}. They also proved that the crank statistic has the same limiting moments. Thus the limit theorem and proposition apply to the crank.\smallskip

A web search for ``Bernoulli numbers as moments'' leads to work of Pitman and Yor \cite{pitman}. They show the following: let $\beta(t)$ be standard one-dimensional Brownian motion. Let $T$ be the first time that standard three-dimensional Brownian motion hits the unit ball ($T$ is independent of $\beta(t)$). Then $\beta(T)$ has odd moments zero and
\begin{equation*}
E\left(\beta(T)^{2l}\right)=\left(2^{2l}-2\right)|B_{2l}|.
\end{equation*}
This shows that $F_r(x)=P\{\beta(T)\leq x\}$. This strains intuition; what is the connection between Brownian motion and partitions? Our proposition explains things by showing how $F_r(x)$ has other probabilistic interpretations.

\item A simple calculation shows that the $s$th absolute moment of the logistic limit distribution $F_r$ is $2\Gamma(s+1)\pi^{-s}(1-2^{1-s})\zeta(s)$, for any $s>0$. It follows, since uniform integrability is implied by (ii), that
\begin{equation*}
\frac{\sum_m|m|^sN(n,m)}{(6n)^{s/2}p(n)}\sim2\Gamma(s+1)\pi^{-s}(1-2^{1-s})\zeta(s).
\end{equation*}
For odd integers $s$ this was recently obtained, with an error term, via other methods described in a forthcoming work by 
Bringmann and Mahlburg \cite{bm12}.
\end{enumerate}
\end{remark}
\begin{proof}[Proof of Proposition \ref{prop:result}]
From \cite{pitman},
\begin{equation}
E\left(e^{it\beta(T)}\right)=\frac{t}{\sinh(t)}.
\label{4}
\end{equation}
If $W$ has distribution function $e^{-e^{-t}}$ and so density $e^{-e^{-t}-t}$ on $\mathbb{R}$,
\begin{equation*}
E(e^{itW})=\int_{-\infty}^\infty e^{itx-e^{-x}-x}\ dx=\int_0^\infty y^{-it}e^{-y}\ dy=\Gamma(1-it).
\end{equation*}
Thus $E(e^{it(W_1-W_2)})=\Gamma(1-it)\Gamma(1+it)=it\Gamma(it)\Gamma(1-it)=\frac{it\pi}{\sin(it\pi)}=\frac{t\pi}{\sinh(t\pi)}$. Here, the reflection formula $\Gamma(z)\Gamma(1-z)=\pi/\sin(\pi z)$ is used. The distribution of $(W_1-W_2)/\pi$ thus equals the distribution of $\beta(T)$. It is well known and easy to verify that the difference between two independent extreme value variates has the logistic distribution as stated.
\end{proof}
\begin{remark}
One may also study the limiting distribution of $rk(\lambda)$ using the conditioned limit approach of Fristedt \cite{fris}, \cite{pittel}, or the method of Corollary 2.6
of Borodin \cite{borodin}. These avoid moments and give the limiting distribution of many further statistics. Fristedt's approach may be developed to prove the asymptotic independence of $\lambda_1$ and $l(\lambda)$. Combining this, the results of Erd{\"o}s--Lehner, and the calculations of this note gives an alternate proof of the main theorem.
\end{remark}


\end{document}